\documentclass[12pt,leqno]{amsart}
\usepackage{amssymb,amsthm,amsmath,latexsym}
\usepackage[all]{xy}
\usepackage{soul}
\usepackage[normalem]{ulem}
\newcommand{\vfi}{\varphi}

\newtheorem{theorem}{\sc Theorem}[section]
\newtheorem{thm}[theorem]{\sc Theorem}
\newtheorem{lem}[theorem]{\sc Lemma}
\newtheorem{prop}[theorem]{\sc Proposition}
\newtheorem{cor}[theorem]{\sc Corollary}

\newtheorem{rem}[theorem]{\sc Remark}

\newtheorem*{thmA}{Theorem A}
\newtheorem*{thmB}{Theorem B}
\newtheorem*{thmC}{Theorem C}

\title[Non-abelian tensor square of $p$-groups]{Non-abelian tensor square and related  constructions of $p$-groups}
\author[Bastos]{R.  Bastos}
\address{Departamento de Matem\'atica, Universidade de Bras\'ilia,
Brasilia-DF, 70910-900 Brazil}
\email{(Bastos) bastos@mat.unb.br}
\author[de Melo]{E. de Melo}
\email{(de Melo) emersonueg@hotmail.com}
\author[Gon\c calves]{N. Gon\c calves}
\email{(Gon\c calves) nathalia@mat.unb.br}
\author[Nunes]{R. Nunes}
\address{ Departamento de Matem\'atica, Universidade Federal de Goi\'as,
Goi\^ania-GO, 74690-900 Brazil }
\email{(Nunes) ricardo@ufg.br}

\begin{document}

\maketitle

\begin{abstract}
Let $G$ be a group. We denote by $\nu(G)$ a certain extension of the non-abelian tensor square $[G,G^{\varphi}]$ by $G \times G$. We prove that if $G$ is a finite potent $p$-group, then $[G,G^{\varphi}]$ and the $k$-th term of the lower central series  $\gamma_k(\nu(G))$ are potently embedded in $\nu(G)$ (Theorem A). Moreover, we show that if $G$ is a potent $p$-group, then the exponent   $\exp(\nu(G))$ divides $p \cdot \exp(G)$ (Theorem B). We also study the weak commutativity construction of powerful $p$-groups (Theorem C).    
\end{abstract}

%%% ----------------------------------------------------------------------
\maketitle
%%% ----------------------------------------------------------------------
%\tableofcontents

\section{Introduction}

The non-abelian tensor square $G \otimes G$ of a group $G$, as introduced by R. Brown and J.\,L. Loday \cite{BL84, BL}, is defined to be the group generated by all symbols $\; \, g\otimes h, \; g,h\in G$, subject to the relations
\[
gg_1 \otimes h = ( g^{g_1}\otimes h^{g_1}) (g_1\otimes h) \quad
\mbox{and} \quad g\otimes hh_1 = (g\otimes h_1)( g^{h_1} \otimes
h^{h_1})
\]
for all $g,g_1, h,h_1 \in G$, where  we write $x^y$ for the conjugate $y^{-1} x y$ of $x$ by $y$, for any elements $x, y \in G$. In the same paper, R. Brown and J.\,L. Loday show that the third homotopy group of the suspension of an Eilenberg-MacLane space $K(G,1)$ satisfies
$$
\pi_3(SK(G,1)) \cong \mu(G),
$$
where $\mu(G)$ is the kernel of the derived map $G \otimes G \to G'$, given by $g \otimes h \mapsto [g,h]$. Consider the following short exact sequence, as in \cite{BL}, 
$$
1 \to \Delta(G) \to \mu(G) \to H_2(G)\to 1,
$$
where $\Delta(G) = \langle g \otimes g \mid g \in G\rangle$ and $H_2(G)$ is the second homology group of the group $G$. It has been shown in \cite{Miller} that $H_2(G) \cong M(G)$, where $M(G)$ is the Schur multiplier of $G$. See also \cite[Chapter 2 and 3]{NR2}.  

The study of the non-abelian tensor square of groups from a group theoretic point of view was initiated by R. Brown, D.\,L. Johnson and E.\,F. Robertson \cite{BJR}. We observe that the defining relations of the non-abelian tensor square can be viewed as abstractions of commutator relations; thus in \cite{NR1}, N.\,R. Rocco considered the following construction. Let $G$ be a group and let $\varphi : G \rightarrow G^{\varphi}$ be an isomorphism ($G^{\varphi}$ is a copy of $G$, where $g \mapsto g^{\varphi}$, for all $g \in G$). Define the group $\nu(G)$ to be \[ \nu (G):= \langle 
G \cup G^{\varphi} \ \vert \ [g_1,{g_2}^{\varphi}]^{g_3}=[{g_1}^{g_3},({g_2}^{g_3})^{\varphi}]=[g_1,{g_2}^{\varphi}]^{{g_3}^{\varphi}},
\; \ g_i \in G \rangle .\]

The motivation for studying $\nu(G)$ is the commutator connection: indeed, the map  $\Phi: G \otimes G \rightarrow [G, G^{\varphi}]$,
defined by $g \otimes h \mapsto [g , h^{\varphi}]$, for all $g, h \in G$, is an isomorphism  \cite[Proposition 2.6]{NR1}. Therefore, from now on we shall  identify the non-abelian tensor square $G \otimes G$ with the subgroup $[G,G^{\vfi}]$ of $\nu(G)$ and write $[g, h^\vfi]$ instead of $g \otimes h$, for all 
$g, h \in G.$ For a fuller treatment we refer the reader to \cite{K,NR01}. 

Our purpose is to study the structure of the non-abelian tensor square and related constructions of finite powerful and potent $p$-groups.

Let $p$ be a prime number. A finite $p$-group $G$ is said to be {\em powerful} if $p>2$ and $G' \leqslant G^p$, or $p=2$ and $G' \leqslant G^4$. We can define a more general class of $p$-groups. We call a finite $p$-group {\em potent} if $p>2$ and $\gamma_{p-1}(G) \leqslant G^p$, or $p=2$ and $G' \leqslant G^4$. Note that the family of potent $p$-groups contains all powerful $p$-groups. Recall that a subgroup $N$ of $G$ is potently embedded in $G$ if $[N, G]\leqslant N^4$, for $p=2$, or $[N,_{p-2}G]\leqslant N^p$ for odd prime  ($N$ is powerfully embedded in $G$ if $[N, G]\leqslant N^4$, for $p=2$, or $[N,G]\leqslant N^p$ for odd prime). More information on finite powerful and potent  $p$-groups can be found in \cite{D} and  \cite{JJ}, respectively.

In \cite{M09}, Moravec proved that if $G$ is a powerful $p$-group, then the non-abelian tensor square  $[G,G^{\varphi}]$ and the derived subgroup $\nu(G)'$ are powerfully embedded in $\nu(G)$. Moreover, the exponent $\exp(\nu(G)')$ divides $\exp(G)$. We extend these results to potent $p$-groups.  

\begin{thmA}
Let $p$ be a prime and $G$ a finite potent $p$-group. 
\begin{itemize} 
\item[(a)] The non-abelian tensor square $[G,G^{\varphi}]$ is potently embedded in $\nu(G)$;   
\item[(b)] If $k \geqslant 2$, then the $k$-th term of the lower central series $\gamma_k(\nu(G))$ is potently embedded in $\nu(G)$.
\end{itemize}
\end{thmA}

\begin{thmB}
Let $p$ be a prime and $G$ a $p$-group with $\exp(G)=p^e$.
\begin{itemize}
\item[(a)] If $G$ is potent, then $\exp(\nu(G))$ divides $p^{e+1}$;
\item[(b)] If $\gamma_{p-2}(G)\leq G^p$, then $\nu(G)$ is a potent $p$-group. In particular, $\exp(\nu(G)) =p^e$.     
\end{itemize}
\end{thmB}

The following results are immediate consequences of Theorem B. 

\begin{cor}
Let $p \geqslant 5$ be a prime and $G$ a finite powerful $p$-group. Then $\nu(G)$ is a potent $p$-group. In particular,  $\exp(\nu(G))=\exp(G)$.  
\end{cor}

\begin{cor}
Let $p$ be a prime and $G$ a finite potent $p$-group. Then the $\exp(M(G))$ and $\exp(\mu(G))$ divide $p \cdot \exp(G)$.
\end{cor}

Now, we study the weak commutativity construction of powerful $p$-groups. As before,  $G^{\varphi}$ denotes an isomorphic copy of 
$G$ via $\varphi : G \rightarrow
G^{\varphi}$, $g \mapsto g^{\varphi}$, for all $g \in G$. The following group construction was introduced
and studied in \cite{Sidki} $$ \chi(G) = \langle G \cup G^{\varphi} \mid [g,g^{\varphi}]=1, \forall g \in G \rangle.$$ The weak commutativity group $\chi (G)$ maps onto $G$ by $g \mapsto g$, $g^{\varphi }\mapsto g$ with kernel $L(G)=\left\langle g^{-1}g^{\varphi } \mid g\;\in G\right\rangle$, and it maps onto $G \times G$ by $g\mapsto \left( g,1\right) ,g^{\varphi}\mapsto \left( 1,g\right) $ with kernel $D(G)=[G,G^{\varphi }]$. It is an important fact that $L(G)$ and $D(G)$ commute. Define $T(G)$ to be the subgroup of $G\times G\times G$ generated by $\{(g,g,1),(1,g,g)\mid g\in G\}$. Then $\chi (G)$ maps onto $T(G)$ by $g\mapsto \left( g,g,1\right)$, $g^{\varphi }\mapsto \left( 1,g,g\right) $, with kernel $W(G)=L(G)\cap D(G)$, an abelian group. In particular, the quotient $\chi(G)/W(G)$ is isomorphic to a subgroup of $G \times G \times G$. A further normal subgroup of $\chi(G)$ is $R(G) = [G,L(G),G^{\varphi}]$, where the quotient $W(G)/R(G)$ is isomorphic to the Schur multiplier $M(G)$. Moreover, in \cite{NR1,NR2}, it was proved that the constructions $\chi(G)$ and $\nu(G)$ have large isomorphic quotients. More precisely, $$\dfrac{\nu(G)}{\Delta(G)} \cong \dfrac{\chi(G)}{R(G)}.$$ 
See \cite[Remark 2]{NR1} and \cite[Remark 4]{NR2}) for more details. The group $\chi(G)$ inherits many properties of the argument $G$; for instance, if $G$ is a finite $\pi$-group ($\pi$ a set of primes), nilpotent, solvable, locally nilpotent, or polycyclic-by-finite group, then so is $\chi(G)$ \cite{Sidki,Roc82,GRS,LO}.  

The following result is an extension of Moravec's results \cite{M09} in the context of the weak commutativity construction.    

\begin{thmC} 
Let $p$ be an odd prime and $G$ a powerful $p$-group. Then the subgroups $\chi(G)'$ and  $D(G)$  are powerfully embedded in  $\chi(G)$.
\end{thmC}

The above theorem is no longer valid if we drop the assumption that $p$ is an odd prime (see Remark \ref{rem.odd}, below).

The paper is organized as follows. In Section 2 we summarize without proofs some results on finite $p$-groups. In the third section we prove Theorems A and B. The proof of Theorem C is given in  Section 4.

\section{Finite $p$-groups}

Let $G$ be a $p$-group. Consider the following subgroups of $G$: $\Omega_n(G)=\langle g\in G \ | \ g^{p^n}=1 \rangle$ and $G^{p^{n}}=\langle g^{p^{n}} \ | \ g\in G \rangle$. If $G$ is abelian, we have the following simple description of these subgroups:
\begin{enumerate}
    \item[(1)]  $\Omega_n(G)=\{ g\in G \ | \ g^{p^n}=1 \}$.
    \item[(2)]  $G^{p^{n}}=\{ g^{p^{n}} \ | \ g\in G \}$.
\end{enumerate}
Moreover, we have that:
\begin{enumerate}
     \item[(3)]  $|G^{p^{n}}|=|G:\Omega_n(G)|$ for all $n$.
\end{enumerate}

A finite $p$-group $G$ is called power abelian if it satisfies these three conditions for all $n$. It is a very interesting problem to determine which groups are power abelian. Recently, it was proved in \cite{JJ} that potent $p$-groups are also power abelian, for odd primes.

The next result are basic facts about finite $p$-groups (see for instance  \cite{JJ}).

\begin{lem}\label{lem.normal}
Let $G$ be a finite $p$-group and $N$, $M$ normal subgroups of $G$. If $N\leq 
M[N,G]N^p$ then $N \leq M$
\end{lem}

The following theorem is known as P. Hall's collection formula.

\begin{thm}
Let $G$ be a $p$-group and $x, y$ elements of $G$. Then for any $k\geq 0$ we have  
$$(xy)^{p^k}\equiv x^{p^k}y^{p^k} (\text{ mod }\gamma_{2}(L)^{p^k}\gamma_{p}(L)^{p^{k-1}}\gamma_{p^2}(L)^{p^{k-2}}\gamma_{p^3}(L)^{p^{k-3}}\cdots \gamma_{p^k}(L)),$$
where $L=\langle x,y\rangle$. We also have that 
$$[x,y]^{p^k}\equiv [x^{p^k}, y] (\text{ mod }\gamma_{2}(M)^{p^k}\gamma_{p}(M)^{p^{k-1}}\gamma_{p^2}(M)^{p^{k-2}}\ldots \gamma_{p^k}(M)),$$
where $M=\langle x,[x,y]\rangle$.
\end{thm}

The next result is a consequence of P. Hall's formula.
\begin{lem} \label{lem.hallformula}
Let $G$ be a finite $p$-group and $N$, $M$ normal subgroups of $G$. Then $[N^p,M]\leq [N,M]^p[M,_pN]$.
\end{lem}

In \cite[ Theorem 2.1]{JJ}, they prove the following useful lemma.
\begin{lem}\label{lem1}
If $G$ is a potent $p$-group, then $\gamma_{k+1}(G)\leqslant \gamma_k(G)^4$, for $p=2$, and $\gamma_{p-1+k}(G)\leqslant(\gamma_{k+1}(G))^p$, for $p\geq3$. 
\end{lem}

\section{Proofs of Theorems A and B}

For the convenience of the reader we repeat some relevant definitions in the context of the non-abelian tensor square (cf.  \cite{NR2}). Recall that there is an epimorphism $\rho: \nu(G) \to G$, given by $g \mapsto g$, $h^{\vfi} \mapsto h$, which induces the derived map $\rho':[G,G^{\vfi}] \to G'$, $[g,h^{\vfi}] \mapsto [g,h]$, for all $g,h \in G$. In the notation of \cite[Section 2]{NR2}, let $\mu(G)$ denote the kernel of $\rho'$, a central subgroup of $\nu(G)$. In particular, $$\dfrac{[G,G^{\vfi}]}{\mu(G)} \cong G'.$$

The next proposition will be used in the proofs of Theorems A and B. 

\begin{prop}\label{prop.tec}
Let $p\geq3$ be a prime and $n\in \mathbb{N}$ such that $1<n<p$. Suppose that $G$ is a finite $p$-group such that $\gamma_n(G)\leqslant G^p$. Then $\gamma_{n+1}(\nu(G))\leqslant\gamma_2(\nu(G))^p$.
\end{prop}

\begin{proof}
Clearly, we can assume that $(\gamma_2(\nu(G)))^{p}=[\nu(G), \nu(G)]^p=1$. In particular, we have that $[G, G]^p=[G, G^\varphi]^p=1$. 

As $\gamma_n(G)\leqslant G^p$, we obtain that $\gamma_{n+1}(G)\leqslant[G^p, G]$. On the other hand, by the Phillip Hall's formula we know that $$[G^p, G]\equiv[G, G]^p  \ (\text{mod}\ \gamma_{p+1}(G)).$$
		
Now, by Lemma \ref{lem1}  $\gamma_{p+1}(G)\leqslant \gamma_{2}(\nu(G))^p$. Therefore $[G^p, G]\leqslant [G, G]^p=1$. As a consequence the nilpotency class of $G$ and $G^\varphi$ is at most $n$. Recall that $\gamma_{n+2}(\nu(G))=\gamma_{n+2}(G)\gamma_{n+2}(G^{\varphi})[\gamma_{n+1}(G),G^{\varphi}]$ by \cite[Theorem 3.1]{NR1}. Thus, we obtain that the nilpotency class of $\nu(G)$ is at most $n+1$.
			
We will show that in fact $\gamma_{n+1}(\nu(G))=1$. Again, by \cite[Theorem 3.1]{NR1}, $$\gamma_{n+1}(\nu(G))=\gamma_{n+1}(G)\gamma_{n+1}(G^{\varphi})[\gamma_n(G), G^{\varphi}].$$
			
Then $\gamma_{n+1}(\nu(G))=[\gamma_{n}(G), G^{\varphi}]\leqslant [G^p, G^{\varphi}]$. It is sufficient to prove that each generator of $[G^p, G^\varphi]$ is trivial. Let $x\in G$ and $y^\varphi\in G^{\varphi}$. Using the Phillip Hall's formula we have that 
$$[x^p, y^\varphi] \equiv\ [x, y^\varphi]^p (\text{mod}\ \gamma_2(M)^p\gamma_p(M)),$$
where $M=\langle x, [x, y^\varphi]\rangle$.
			
Note that $\gamma_p(M)\leqslant\gamma_{p+1}(\nu(G))$. Hence $\gamma_2(M)^p\leqslant \gamma_{2}(\nu(G))^p=1$ and $\gamma_p(M)\leqslant\gamma_{p+1}(\nu(G))\leqslant\gamma_{n+2}(\nu(G))=1$, since $p+1\geq n+2$.
			
Therefore $[x^p, y^{\varphi}]=1$ for any  $ x, y \in G$ and so $[G^p, G^{\varphi}]=1$ as desired.
\end{proof}

We are now in a position to prove Theorem A. 

\begin{proof}
(a). Recall that the definitions of powerful and potent $p$-groups coincide for $p=2$ and $3$. Using  Moravec's result \cite{M09}, the non-abelian tensor square $[G,G^{\varphi}]$ is powerful embedded in $\nu(G)$, when $G$ is a powerful $p$-group. Now, it remains to consider potent $p$-groups with $p\geq 5$. 

Hence we need to prove that $[[G, G^\varphi], _{p-2} \nu(G)]\leqslant [G, G^\varphi]^p$. Suppose that $[G, G^\varphi]^p=1$. Consider the derived map $\rho': [G,G^{\varphi}] \to G'$, given by $[g,h^{\varphi}] \mapsto [g,h]$. Since $[G, G^\varphi]^p=1$ and $Im(\rho')=G'$, we deduce that $(G')^p=1$. By Lemma \ref{lem1},  $\gamma_p(G)=1=\gamma_p(G^\varphi)$.
    
    Observe that $[[G, G^\varphi], _{p-2} \nu(G)]\leqslant \gamma_p(\nu(G))$ and, by \cite[Theorem 3.1]{NR1},  $\gamma_p(\nu(G))=\gamma_p(G)\gamma_p(G^\varphi)[\gamma_{p-1}(G), G^\varphi]$.   Therefore $\gamma_{p}(\nu(G))=[\gamma_{p-1}(G), G^\varphi]\leqslant [G^p, G^\varphi]$.

Let $x\in G$ and $y^\varphi\in G^{\varphi}$. Using the Phillip Hall's formula we have that 
$$[x^p, y^\varphi] \equiv\ [x, y^\varphi]^p (\text{mod}\ \gamma_2(M)^p\gamma_p(M)),$$
where $M=\langle x, [x, y^\varphi]\rangle$. Note that $\gamma_p(M)\leqslant\gamma_{p+1}(\nu(G))$. Hence $\gamma_2(M)^p\leqslant \gamma_{2}(\nu(G))^p=1$ and $\gamma_p(M)\leqslant\gamma_{p+1}(\nu(G))=1$. Consequently, $[x^p, y^{\varphi}]=1$ for any  $ x, y \in G$ and so $[G^p, G^{\varphi}]=1$. Therefore $[[G, G^\varphi], _{p-2} \nu(G)]=1$, as wished. \\  

(b). We will prove by induction on $k$. For $k=2$ apply Proposition \ref{prop.tec} for $n=p-1$. Suppose by induction hypothesis that $[\gamma_{k}(\nu(G)), _{p-2}\nu(G)] \leqslant (\gamma_{k}(\nu(G)))^p$. By Lemma \ref{lem.hallformula}, 
$$\begin{array}{lcl}
		[\gamma_{k+1}(\nu(G)), _{p-2}\nu(G)]& \leqslant & [(\gamma_k(\nu(G)))^p, \nu(G)]\\
		    \ & \leqslant & [\gamma_k(\nu(G)), \nu(G)]^p[\nu(G),\  _p\ \gamma_k(\nu(G))]\\
		    \ & \leqslant & (\gamma_{k+1}(\nu(G))^p[\gamma_{k+1}(\nu(G)),\  _{p-2}\ \nu(G), \nu(G)].
\end{array}$$
Therefore, by Lemma \ref{lem.normal} we have  $$[\gamma_{k+1}(\nu(G)), _{p-2}\nu(G)]\leqslant (\gamma_{k+1}(\nu(G))^p,$$ which is the desired conclusion.
\end{proof}

\begin{proof}[Proof of Theorem B] \  
\noindent (a). In \cite{JJJ} it was proved that if $P$ is a finite $p$-group such that $\gamma_{k(p-1)}(P)\leq \gamma_r(P)^{p^s}$ for some $r$ and $s$ such that $k(p-1)<r+s(p-1)$, then the exponent of $\Omega_i(P)$ is at most $p^{i+k-1}$ for all $i$.

Note that $\gamma_{2(p-1)}(\nu(G))=[\gamma_p(\nu(G)),_{p-2}\nu(G)]$. Then using Theorem A (b) we conclude that $\gamma_{2(p-1)}(\nu(G)) \leq \gamma_{p}(\nu(G))^p$. Now, by \cite{JJJ} we have that $\exp(\Omega_e(\nu(G)))=p^{e+1}$. In particular, $\exp(\nu(G))=p^{e+1}$ since $\nu(G)$ is generated by $G$ and $G^{\varphi}$. \\

\noindent (b). Since $\gamma_{p-2}(G) \leqslant G^{p}$, by Proposition \ref{prop.tec}, we conclude that $\gamma_{p-1}(\nu(G))\leq [\nu(G),\nu(G)]^p$ and so $\nu(G)$ is a potent $p$-group. By \cite{JJ} potent $p$-groups are power abelian $p$-groups, whenever $p$ is an odd prime. As $\nu(G)$ is generated by $G$ and $G^{\varphi}$ we have $\exp(\nu(G)) = \exp(G)$. The proof is complete. 
\end{proof}

\section{Proof of Theorem C}

For the convenience of the reader we repeat some relevant definitions in the context of the weak commutativity construction (cf.  \cite{Sidki}). Let $G$ be a group. The weak commutativity construction of $G$ is defined as $\chi(G) = \langle G \cup G^{\varphi} \mid [g,g^{\varphi}]=1, \forall g \in G \rangle$. Consider the subgroup $D(G) = [G,G^{\varphi}]$ of $\chi(G)$. 

\begin{thmC} 
Let $p$ be an odd prime and $G$ a powerful $p$-group. Then the subgroups $\chi(G)', \; D(G)$  are powerfully embedded in  $\chi(G)$.
\end{thmC}

\begin{proof}
We need to prove that  $$\gamma_3(\chi(G)) \leqslant [\chi(G),\chi(G)]^p.$$

Clearly, we can assume that $(\gamma_2(\chi(G)))^{p}=[\chi(G), \chi(G)]^p=1$. Moreover, by Lemma \ref{lem.normal}, we can suppose that $(\gamma_2(\chi(G)))^p\gamma_4(\chi(G))=1$. In particular, we have that $[G^\varphi, G^\varphi]^p = [G, G]^p=[G, G^\varphi]^p=1$. Since $G$ is powerful, it follows that $G$ and $G^{\varphi}$ are nilpotent groups of class at most 2. By \cite[Lemma 3.2.1]{Roc82}, $\gamma_3(\chi(G)) \leqslant [G',G^\varphi].$
Since $G$ is a powerful $p$-group, we deduce that 
$$ \gamma_3(\chi(G)) \leqslant [G',G^\varphi] \leqslant [G^p, G^\varphi].$$

Let $x\in G$ and $y^\varphi\in G^{\varphi}$. Using the Phillip Hall's formula we have that 
$$[x^p, y^\varphi] \equiv\ [x, y^\varphi]^p (\text{mod}\ \gamma_2(M)^p\gamma_p(M)),$$
where $M=\langle x, [x, y^\varphi]\rangle$. Note that $\gamma_p(M)\leqslant\gamma_{p+1}(\chi(G))$. Hence $\gamma_2(M)^p\leqslant \gamma_{2}(\chi(G))^p=1$ and $\gamma_p(M)\leqslant\gamma_{p+1}(\chi(G))\leqslant\gamma_{4}(\chi(G))=1$.

Consequently,  $[x^p, y^{\varphi}]=1$ for any  $ x, y \in G$ and so $[G^p, G^{\varphi}]=1$, this establishes that $\chi(G)'$ is powerfully embedded in $\chi(G)$.   

It remains to prove that $D(G)$ is powerfully embedded in $\chi(G)$, that is, $[D(G), \chi(G)] \leqslant D(G)^p .$
Assuming  that $D(G)^p=1$, by  \cite[Proposition 4.1.4]{Sidki} we get that $[G,G]^p=1$, which follows as  above.
\end{proof}

\begin{rem}\label{rem.odd}
Theorem C does not hold assuming $p$=2. Taking $G=\left<a,b,c\right>$ being an elementary abelian 2-group of rank 3, we have that $\chi(G)$ is a nilpotent group of class 3, where
$$D(G) = \left<[a,b^\varphi], [a,c^\varphi], [b,c^\varphi], [a,b^\varphi,c]\right> \cong \mathbb{Z}_2^4$$ and
$$[D(G),G] = \left<[a,b^\varphi,c]\right> \cong \mathbb{Z}_2.$$
\end{rem}

For $p\geqslant 5$ we obtain the following bound to the exponent $\exp(\chi(G))$, when $G$ is a powerful group.

\begin{cor}
Let $G$ be finite powerful $p$-group with $p \geqslant 5$. Then $\chi(G)$ is potent. Moreover, the exponent $\exp(\chi(G))$ divides $\exp(G)$.  
\end{cor}

\begin{proof}
By Theorem C,  $[\chi(G)',\chi(G)] \leqslant \chi(G)^p$. In particular,

$$
\gamma_{p-1}(\chi(G)) \leqslant [\chi(G)',\chi(G)] \leqslant \chi(G)^p.
$$

As $p\geqslant 5$, we have $G$ is a potent $p$-group and so, the exponent $\exp(\chi(G)) = \exp(G)$, which completes the proof.  
\end{proof}

\section*{Acknowledgment}
The authors wish to thank professor Nora\'i Romeu Rocco for interesting discussions. This work was partially supported by FAPDF - Brazil, Grant: 0193.001344/2016.

\end{document}